\documentclass[11pt]{article}
\usepackage[latin1]{inputenc}

\usepackage[a4paper,scale={.82,.8}]{geometry}
\usepackage[usenames, dvipsnames]{xcolor}
\usepackage[colorlinks=true, pdfstartview=FitV,linkcolor=ForestGreen,citecolor=ForestGreen, urlcolor=blue]{hyperref}

\makeatletter
\renewcommand\subparagraph{\@startsection {subparagraph}{5}{\z@ }{3.25ex \@plus 1ex
 \@minus .2ex}{-1em}{\normalfont \normalsize \bfseries }}%
\makeatother

\usepackage{amsmath,amssymb,amsfonts,lipsum}

\newcommand{\N}{\mathbb{N}}
 
\newcommand{\R}{\mathbb{R}}
 
\newcommand{\T}{\mathbb{T}} 
\renewcommand{\div}{\operatorname{div}}

\newcommand{\eg}[1][~]{\textit{e.g.#1}}
\newcommand{\ie}[1][~]{\textit{i.e.#1}}

\newtheorem{thm}{Theorem}

\newtheorem{prop}[thm]{Proposition}

\newtheorem{dfn}[thm]{Definition}
\newcommand{\proof}{\paragraph{\it Proof.}}
\newcommand{\cqfd}{\hfill\rule{1ex}{1ex}}
\newtheorem{rmk}[thm]{Remark}

\newcommand{\atom}{\mathcal{A}}
\newcommand{\norme}[2][]{\| #2 \|_{#1}}

\newcommand{\ball}[2]{\mathcal{B}(#1,#2)}
\DeclareMathOperator{\BMO}{BMO}

\title{Variation on a theme by Kiselev and Nazarov:\\
H\"older estimates for non-local transport-diffusion,\\
along a non-divergence-free BMO field.}
\author{Ioann Vasilyev\footnote{Email: \small\texttt{Ioann.vasilyev@u-pec.fr}} \: and \,
Fran\c{c}ois Vigneron\footnote{Email: \small\texttt{francois.vigneron@u-pec.fr}}\\[1ex]
\small Universit\'e Paris-Est\\ 
\small LAMA (\textsc{umr8050}), UPEC, UPEM, CNRS \\
\small 61, avenue du G\'en\'eral de Gaulle, F94010 Cr\'eteil, France.}

\begin{document}
\maketitle

\begin{abstract}
We prove uniform H\"older regularity estimates for a transport-diffusion equation with a fractional diffusion operator,
and a general advection field in BMO,
as long as the order of the diffusion dominates the transport term at small scales;
our only requirement is the smallness of the negative part of the divergence
in some critical Lebesgue space. In comparison to a celebrated result by L.~Silvestre~(2012), our advection field does not need to be bounded.
A similar result can be obtained in the super-critical case if the advection field is H\"older continuous.
Our proof is inspired by A. Kiselev and F. Nazarov~(2010) and is based on the dual evolution technique.
The idea is to propagate an atom property (\textit{i.e.} localization and integrability in Lebesgue spaces)
under the dual conservation law, when it is coupled with the fractional diffusion operator.
\\[1ex]
\textbf{MSC primary:} 35B65.\\
\textbf{MSC secondary:}  35R11, 35Q35.
\end{abstract}

In this article, we are interested in the following transport-diffusion equation:
\begin{equation}\label{TDalpha}
\begin{cases}
\partial_t \theta + (-\Delta)^{\alpha/2} \theta = (v \cdot \nabla) \theta\\
\theta(0,x) = \theta_0(x).
\end{cases}
\end{equation}
We consider this Cauchy problem where $\theta:[0,T]\times \R^d\to \R$ is unknown. The vector field $v$ is given and
is of bounded mean oscillation, \ie it belongs to the space $\BMO(\R^d)$. In what follows, we will \textbf{not} assume that $\div v=0$.
We will also consider the periodic problem on $\T^d$.

We are interested in the whole range of parameters $0<\alpha  < 2$ and in particular in the critical non-local diffusion (\ie $\alpha=1$)
where diffusion and transport are of similar order.
We will, in passing, consider the classical local case $\alpha=2$; however, the local case is much better
understood and  we refer the reader, for exemple, to the recent monograph~\cite{LBL2019} and the references therein.

For $\alpha>1$, the equation is called sub-critical because the drift is then of lower order than the diffusion,
which means that the diffusion will be stronger at the smallest scales. On the contrary, for $\alpha<1$,
the drift will be stronger than the diffusion at smaller scales and the equation is called super-critical.
However, as the diffusion operator remains invariant under Galilean transforms, one can still expect a mild smoothing
effect, as long as $\alpha>0$, because putative advected singularities cannot ride along a ``defect'' of the diffusion operator,
simply because the homogeneous and isotropic operator does not have any.

\bigskip
The fractional derivative $(-\Delta)^{\alpha/2}$ is the Fourier multiplier by  $|\xi|^{\alpha}$.
It admits the following kernel expansion on $\R^d$ (see~\eg\cite[\S2]{CC2004}):
\begin{equation}\label{kernelPseudoLaplace}
(-\Delta)^{\alpha/2} \theta (x) =  -c_{d,\alpha}
\operatorname{pv} \left( \int_{\R^d} \frac{\theta(y)-\theta(x)}{|y-x|^{d+\alpha}} dy \right)
\end{equation}
with $c_{d,\alpha}>0$.
As $(-\Delta)^{\alpha/2}$ is a non-local derivative of order $\alpha$ that does not induce phase-shift, it performs a sort of ``graphical'' interpolation
between the graph of $\theta$ and that of $-\Delta \theta$.

\bigskip
The goal is to establish uniform H\"older regularity estimates of the solution of~\eqref{TDalpha} at a later time~$t>0$. 
As demonstrated by A.~Kiselev and F.~Nazarov~\cite{KN2009} and M.~Dabkowski~\cite{D2010} for $\alpha=1$, an elegant but powerful technique consists in using the atomic characterization of H\"older classes (see below). Multiplying the equation by a test-function $\psi(t-s)$ where $\psi$ solves the dual evolution equation, one can then
exchange an atomic estimate of $\theta(t)\psi_0$ against an atomic estimate of $\theta_0\psi(t)$.
Provided that the dual evolution propagates the atomic property in a controlled way, one then gets the desired H\"older regularity.

\bigskip
Another powerful approach, used \eg by L.A.~Caffarelli and A.~Vasseur \cite{CV2010} and L.~Silvestre~\cite{Silv2010},
consists in using the parabolic De Giorgi method. The idea is to first establish
a set of energy estimates. Then, one uses the rigidity given by the equation to deduce scalable uniform bounds. More precisely,
one computes the energy cost of one oscillation of the solution between its maximum and minimum value. Given that the total amount
of energy available is finite, this limits the size of the oscillations that the solution can perform at a given scale, which ultimately translates
as regularity in the H\"older classes.

\bigskip
Finally, the method of the modulus of continuity was successfully introduced by A.~Kiselev, F.~Nazarov and A.~Volberg~\cite{KNV2007}
for SQG and later used by L.~Silvestre, V.~Vicol~\cite{SV2018} to study related PDEs.

\bigskip
The case of a divergence-free transport field in~\eqref{TDalpha} is of particular interest as it closely relates to the quasi-geostrophic equation
where $d=2$ and $u = \nabla^\perp (-\Delta)^{-1/2} \theta$; see~\cite{CCW2001}, \cite{CV2010}, \cite{KNV2007}.
At a technical level, the computations of~\cite{KN2009}, \cite{D2010} rely  in a crucial way on the assumption that $\div v=0$.
First, as they use the non-conservative form for the dual evolution equation
\[
\partial_s \psi + (-\Delta)^{\alpha/2} \psi = - (v(t-s) \cdot \nabla) \psi
\]
each integration by part requires a divergence-free field.
Second, they invoque the maximum-principle and the decay of the $L^p$-norms for this kind of equation, which was established
by A.~Cordoba, D.~Cordoba~\cite{CC2004}, but either for the particular vector field of SQG, or for a divergence-free transport field.

\bigskip
The key of our article is that we allow $\div v \neq 0$.
The sign of $\div v$ then becomes crucial because convergent characteristics will tend to create shocks and break down
the regularity while divergent characteristics (rarefaction waves) have a natural smoothing effect and tend to reduce steep gradients.
The presence of the diffusion operator fixes the issues of either non-existence or non-uniqueness that the pure transport would induce.
In that context, the pertinent question is thus to establish quantitative estimates of the regularity norms.

\bigskip
Our main result is the following statement that provides uniform H\"older bounds for solutions of~\eqref{TDalpha}. 
\begin{thm}\label{mainThm}
We consider either $\R^d$ or $\T^d$ as the ambiant space, which is of dimension $d>\alpha$,
and an advection field
\begin{equation}
\begin{cases}
v\in C^{1-\alpha} & \text{if}\quad 0<\alpha<1,\\
v\in\BMO & \text{if}\quad 1\leq \alpha\leq 2.
\end{cases}
\end{equation}
In both cases, one requires:
\begin{equation}\label{critical}
\norme[L^{d/\alpha}]{(\div v)_-} \leq S_{\alpha/2} = \sup_{f\in \dot{H}^{\alpha/2}}\left( \frac{ \displaystyle \int |(-\Delta)^{\alpha/4}f|^2 }{ \norme[L^{2d/(d-\alpha)}]{f}^2 } \right)
\end{equation}
where $(\div v)_-$ is the negative part of the divergence (see~\eqref{PosNegNotation} below) and $\dot{H}^{\alpha/2}$ denotes the homogeneous
Sobolev space on $\R^d$, or the average-free space on $\T^d$.
There exist $\beta>0$ and $C>0$ that depend solely on $d$, $\alpha$ and respectively either on~$\norme[\BMO]{v}$ or $\norme[C^{1-\alpha}]{v}$ such that,
for any $\theta_0\in L^q$ with $2\leq q \leq \infty$,
the corresponding solution of~\eqref{TDalpha} satisfies:
\begin{equation}\label{mainReg}
\forall t\in (0,1], \qquad
\norme[C^\beta]{\theta(t,\cdot)}
\leq C t^{-(\beta+\frac{d}{q})/\alpha}  \norme[L^q]{\theta_0}.
\end{equation}
Moreover, if $\theta_0\in C^\beta$, then:
\begin{equation}\label{mainProp}
\forall t\geq0, \qquad
\norme[C^\beta]{\theta(t,\cdot)}
\leq C' \norme[C^\beta]{\theta_0} 
\end{equation}
for some constant $C'$.
\end{thm}

Among the results quoted in this introduction, that of L.~Silvestre~\cite{Silv2010} is the only one that allows for non-divergence-free fields,
so we will focus on this one.
For $\alpha\geq1$, he assumes that $v\in L^\infty(\R^d)$ but requires no other size constraint on $v$; for $\alpha<1$, he assumes that
$v\in C^{1-\alpha}$ in order to compensate for the super-criticality of the equation. In both cases, 
L.~Silvestre proves the H\"older continuity of the solution of~\eqref{TDalpha} at positive times, namely that:
\[
\norme[C^\beta(\R^d)]{\theta(t)} \leq C t^{-\beta/\alpha} \norme[L^\infty(\R^d)]{\theta_0}
\]
and he also establishes a similar $C^{\beta/\alpha}$-regularity estimate time-wise.

\medskip
Our result is complimentary and does not quite compare to that of L.~Silvestre when $\alpha\geq1$ because, in that case, we do not assume
that $v$ is bounded.
Instead, we assume that $v$ belongs to $\BMO$ and we require the negative part of its divergence to be small.
However, we do not constrain the size of~$(\div v)_+$ and, in particular any BMO divergence-free field is admissible in our result.

It is likely that a ``universal'' result  holds true for a general advection field $v=v_1+v_2$ with $v_1\in L^\infty$ and $v_2\in \BMO$ with $(\div v_2)_-$ small.
Results in this spirit are known for general diffusion operators of order $\alpha=2$ (see~\cite{LBL2019}).
However, mixing the DeGiorgi and atomic methods is not obvious.

\begin{rmk}
The critical value in~\eqref{critical} can be slightly relaxed but the constants $\beta$ and $C$ will then also depend on the size of~$(\div v)_-$; see remark~\ref{relax}.  When $\alpha=2$, the dimension $d=2$ is excluded from our statement (and similarly the case $\alpha=d=1$) because the corresponding Sobolev Space $H^{\alpha/2}(\R^d)$
fails to be included in $L^\infty$.
\end{rmk}

Our proof is inspired by A.~Kiselev and F.~Nazarov~\cite{KN2009} and we use the dual evolution technique.
Following M.~Dabkowski~\cite{D2010}, we also use  $L^p$-based atoms for more flexibility.
In order to deal with non-divergence-free transport fields, the key is to replace
the dual evolution equation by a transport-diffusion equation expressed in a \textit{conservative} form, \ie as a conservation law.
The presence of a general exponent of diffusion $\alpha\neq 1$ also brings about some technicalities
that need to be addressed (see the end of~\S\ref{par:regularity}). 

\bigskip
There are many other takes on the question of regularity for transport-diffusion equations.
For example, D.~Bresch, P.-E.~Jabin \cite{BJ2018} have studied the regularity of weak solutions to the advection equation set in conservative form.
In particular, they investigated the case where both $v$ and $\div v$ belong to some
Lebesgue space, which allows point-wise unbounded variations of the field and of its compressibility.

Another related study, albeit slightly more distant from~\eqref{TDalpha}, is that of the kinetic Fokker-Planck equation.
On that mater, we refer the reader  to C.~Mouhot \cite{M2018} and the references therein,
and, \eg[,] to the regularity result by F.~Golse, A.~Vasseur \cite{GV2015}.

\bigskip
The role of fractional diffusion operators in physical models is growing (see \eg J.~Va\'{z}quez's book~\cite{Vaz2017}).
As for our motivations, we are interested in studying variants of a non-local Burgers equation introduced by C.~Imbert, R.~Shvydkoy
and one of the present authors in \cite{ISV2016}. The regularity of the solution for unsigned data remains an open problem that
seems to share some nontrivial similarities with hydrodynamic turbulence.

\bigskip
The present article is structured as follows.
In \S\ref{par:atoms}, we recall how H\"older classes can be characterized in terms of atoms.
We draft the general ideas on how regularity can be obtained by studying the dual conservation law in \S\ref{par:regularity}.
A weak version  (\ie in Lebesgue spaces) of the maximum principle is established in \S\ref{par:maxPrinciple}.
It is then put to use in \S\ref{par:propAtom} to quantify how the atomic property is propagated.
The proof of  theorem~\ref{mainThm} and, in particular, how the choices of structural constants should be made, is detailed in \S\ref{par:proof}.

\paragraph{Notations.}
In this article, one uses the following common notations for $a,b,x\in\R$:
\begin{gather}
a\wedge b = \min\{a,b\}, \qquad a\vee b = \max\{a,b\}\\\label{PosNegNotation}
x_+ = a \vee 0, \qquad x_- = (-x)_+, \qquad x = x_+ - x_-, \qquad |x|= x_+ + x_-.
\end{gather}
Balls are denoted by $\ball{x_0}{r}=\left\{x\in\R^d \,;\, |x-x_0|<r\right\}$ where $|\cdot|$ denotes the Euclidian distance on~$\R^d$.

\section{Atomic characterization of H\"older classes}\label{par:atoms}

Throughout the article, the constant $A\gg1$ is fixed, but is chosen arbitrarily large and $\omega\in(0,1)$.
One can check that none of the final estimates actually depend on the particular values of $A$ or $\omega$.
\begin{dfn}
For $r>0$ and $p\in (1,\infty]$, the atomic class $\atom^p_r(\R^d)$ is defined as a subset of $C^\infty(\R^d)$ by the following three conditions:
\begin{gather}
\int_{\R^d} \varphi(x) dx =0,\\
\norme[L^1]{\varphi} \leq 1 \qquad\text{and}\qquad
\norme[L^p]{\varphi} \leq A r^{-d(1-\frac{1}{p})},\\
\exists x_0\in\R^d, \qquad \int_{\R^d} |\varphi(x)| \Omega(x-x_0) dx \leq r^\omega
\end{gather}
where $\Omega(z) = |z|^\omega \wedge 1 \in L^\infty(\R^d)$.
If $\lambda^{-1} \varphi(x) \in \atom^p_r(\R^d)$ for some $\lambda>0$, then one says that 
\begin{equation}
\varphi\in \lambda\cdot \atom^p_r(\R^d).
\end{equation}
\end{dfn}

\begin{rmk}
A typical example of an atom $\atom^p_r$ of radius $r< 1$ is the function $\varphi = \varphi_r\ast \rho_\epsilon$ where $\rho_\epsilon$ is
an standard mollifier supported in $\ball{0}{\epsilon}$ and
\[
\varphi_r(x) = \begin{cases}
- C r^{-d} & \text{if } |x-x_0| \leq r 2^{-1/d}  \\
+ C r^{-d} & \text{if } r 2^{-1/d}  < |x-x_0| \leq r \\
0 & \text{if } |x-x_0|>r
\end{cases}
\]
for small enough constants $C < \min\{ |\ball{0}{1}|^{-1}; A |\ball{0}{1}|^{-1/p}; |\ball{0}{1}|/(d+\omega)\}$ and $\epsilon>0$.
The volume of the unit ball is $|\ball{0}{1}| = \pi^{d/2}/\Gamma\left(\frac{d}{2}+1\right)$.
\end{rmk}

One can control the $L^q$-norm of atoms for $q\in[1,p]$ by a real interpolation estimate.
\begin{prop}\label{AtomL1LpBound}
If $\varphi \in \atom^p_r(\R^d)$, then  for any $1\leq q\leq p$:
\begin{equation}\label{AtomLpBound}
\norme[L^q]{\varphi} \leq A^{\frac{1-1/q}{1-1/p}} r^{-d(1-\frac{1}{q})}.
\end{equation}
\end{prop}
\begin{proof}
Apply the interpolation inequality $\norme[L^q]{f}\leq \norme[L^1]{f}^{1-\theta} \norme[L^p]{f}^{\theta}$
with $\theta = (1-\frac{1}{q})/(1-\frac{1}{p}) \in [0,1]$.
\cqfd\end{proof}

\bigskip
Atoms are the ``poor man's wavelet'' and offer a very comfortable characterization of H\"older's classes.
\begin{prop}
For $0<\beta<1$, a bounded function $f$ belongs to $C^\beta(\R^d)$ if and only if
\begin{equation}\label{atomCbeta}
\sup_{\substack{0<r\leq 1\\ \varphi\in \atom^p_r(\R^d)}} r^{-\beta} \left| \int_{\R^d} f(x) \varphi(x) dx  \right| < \infty
\end{equation}
for  some $p\in(1,\infty]$.
Moreover, the left-hand side of~\eqref{atomCbeta} is equivalent to the usual semi-norm on~$C^\beta(\R^d)$.
\end{prop}
\proof The proof is a classical exercise.
See \cite{KN2009}, \cite{D2010} (resp.~for $p=\infty$ and $p<\infty$) for a short proof that relies on the Littlewood-Paley theory \cite{STEIN1993}. 
The key is the equivalent control of $\sup_{j\in\N} 2^{\beta  j}\norme[L^\infty]{\Delta_j f}$ where $\Delta_j$ is
a smooth projection on the frequency scale of order $2^j$.
\cqfd

\section{Regularity through the dual conservation law}\label{par:regularity}

As explained in the previous section, in order to obtain estimates of the H\"older regularity  of the solution of~\eqref{TDalpha} at a time $t>0$, one
needs to control
\[
r^{-\beta} \int_{\R^d} \theta(t,x) \psi_0(x) dx
\]
where $\psi_0\in \atom^p_r(\R^d)$ and this control needs to be uniform in $0<r\leq 1$. Let us consider the following test function 
that solves the dual evolution problem, set in a \textit{conservative} form:
\begin{equation}\label{dualEQ}
\begin{cases}
\partial_s \psi(s) + (-\Delta)^{\alpha/2} \psi(s) = - \div\left( v(t-s) \psi(s) \right)\\
\psi(0,x)=\psi_0(x).
\end{cases}
\end{equation}
One can then immediately check the following result.
\begin{prop}
If $\theta$ is a smooth solution of~\eqref{TDalpha} and $\psi$ is a smooth solution of~\eqref{dualEQ}, then one has:
\begin{equation}\label{transfert}
\forall t\geq0,\qquad
\int_{\R^d} \theta(t,x) \psi_0(x) dx = \int_{\R^d} \theta_0(x) \psi(t,x) dx.
\end{equation}
\end{prop}
\proof
Let us indeed use $\psi(t-s)$ as a test function for~\eqref{TDalpha}. One gets:
\[
\iint \partial_s \theta(s) \cdot \psi(t-s) + \iint ((-\Delta)^{\alpha/2} \theta(s)) \cdot \psi(t-s) = 
\iint (v(s)\cdot \nabla) \theta(s) \cdot \psi(t-s).
\]
Here and below, all double integrals are computed for $(s,x)\in [0,t]\times\R^d$ unless stated otherwise and the $x$-variable is not made explicit
unless it is absolutely necessary. Integrating by part time-wise in the first integral and space-wise in the other two gives:
\begin{align*}
\iint \theta(s) \cdot (\partial_s\psi)&(t-s) 
+ \left[ \int_{\R^d} \theta(s) \psi(t-s) \right]_{0}^{t}\\
&+ \iint \theta(s) \cdot ((-\Delta)^{\alpha/2} \psi)(t-s) = 
- \iint  \theta(s)  \cdot \div (v(s) \psi(t-s)).
\end{align*}
Thanks to~\eqref{dualEQ}, the double integrals cancel each other out and one is left with \eqref{transfert}.
\cqfd

\begin{rmk}
In~\cite{KN2009}, the authors used the non-conservative dual form $-v(t-s)\cdot\nabla\psi(s)$. This choice was
harmless and perfectly adapted to their purpose because they assumed $v$ to be divergence-free.
Here, on the contrary, it is crucial that the right-hand side of~\eqref{dualEQ} takes the form of a conservation law.
\end{rmk}

For the sake of the argument, let us assume that one will be able to show subsequently (which is indeed the case if $\alpha=1$)
the following infinitesimal propagation property for~\eqref{dualEQ}:
\[
\psi_0\in\atom_r^p
\qquad\Longrightarrow\qquad
\forall s\in [0,\gamma r], \qquad \psi(s) \in (1-h(r) s) \atom_{r+Ks}^p
\]
for a given value of $p\in(1,\infty]$,
with universal constants $\gamma, K$ that do not depend on $r$ nor $\psi_0$ and a universal function $h$.
One can then immediately infer a global propagation property:
\[
\psi_0\in\atom_r^p
\qquad\Longrightarrow\qquad
\forall s\in \left[0, \frac{1-r}{K}\right], \qquad \psi(s) \in f_r(s) \atom_{r+Ks}^p
\]
with $f_r'(s) \geq -h(r+Ks) f_r(s)$. Let us introduce a function $H$ such that $H'(z)=h(z)$. Then
\[
f_r(s) = \exp\left( \frac{H(r)-H(r+Ks)}{K} \right)
\]
is an acceptable bound for the global propagation property.
Coupled with~\eqref{transfert}, this means the following: for any solution of~\eqref{TDalpha} and $\psi_0\in\atom^p_r$, one has
\[
\left| \int_{\R^d} \theta(t,x) \psi_0(x) dx \right| = \left| \int_{\R^d} \theta_0(x) \psi(t,x) dx \right| 
\leq  \norme[C^\beta]{\theta_0} (r+Kt)^\beta f_r(t).
\]
One is thus able to propagate $C^\beta(\R^d)$ bounds of $\theta$ if $f_r(s) \leq C ( \frac{r}{r+Ks} )^\beta$.
This is the case if, for example, $h(r) = \delta/r$ with $\delta = K\beta$.
The regularization estimate is obtained in the same fashion.

\bigskip
Dealing with a general exponent $\alpha$ requires a slightly more careful computation. 
The fundamental idea remains that the dual conservation law propagates atoms and that a small gain on the amplitude of the atoms
can be obtained as a tradeoff with a slight increase in the size of the atoms' radii.

\medskip
The main technical difficulty is that the radii now grow as $(r^\alpha+Ks)^{1/\alpha}$, which is not linear in $s$ anymore,
at least not when $s\gg r^{\alpha}$. This non-linear region invades any neighborhood of $r=s=0$ and
the corresponding correction of amplitude will be $h(r)=\delta/r^\alpha$. We found that the simplest
workaround  is to forfeit the ODE point of view presented here for $\alpha=1$ and to use direct estimates on
the corresponding rate of change during a finite increment of an Euler scheme (see \S\ref{par:propAtom}, estimate~\eqref{radiusMatch}).

\section{Weak maximum principle for the dual conservation law}\label{par:maxPrinciple}
In this section, we establish the weak maximum principle \ie the decay of the $L^p$-norms for
a non-local transport-diffusion equation written in a conservative form. In this section, one considers
thus the following general problem for $0<\alpha\leq2$:
\begin{equation}\label{TDalphaCons}
\begin{cases}
\partial_s \psi(s) + (-\Delta)^{\alpha/2} \psi(s) = - \div\left( \mathbf{v}(s) \psi(s) \right)\\
\psi(0,x)=\psi_0(x)
\end{cases}
\end{equation}
and we will assume, when necessary, that $\displaystyle \int_{\R^d} \psi_0 = 0$.
Subsequently, the results of this section will be applied to~\eqref{dualEQ} at a given time $t>0$ by choosing $\mathbf{v}(s) = v(t-s) \in\BMO(\R^d)$.

\subsection{A brief note on the well-posedness theory}

For smooth $\mathbf{v}(s,x)$, the well-posedness theory of the scalar conservation law
\[
\partial_s \psi(s) = - \div\left( \mathbf{v}(s) \psi(s) \right)
\]
was established by S.N.~Kru\v{z}kov \cite{Kru1970}, in the setting of entropy solutions.
The theory was refined and generalized to the non-conservative convective form
 by R.J.~DiPerna and P.L.~Lions \cite{DipL1989}; their theory ensures that assuming a transport field
$\mathbf{v} \in L^1([0,T];W^{1,1})$ with $(\div \mathbf{v})_-\in L^1([0,T];L^\infty)$ is enough to guarantee
the existence, uniqueness and stability in the proper function spaces.
The key idea is a celebrated commutation lemma:
\[
\rho_\varepsilon \ast ( \mathbf{v}\cdot\nabla) \psi -  \mathbf{v}\cdot\nabla(\rho_\varepsilon \ast \psi)\to 0 \qquad \text{in } L^1([0,T];L^\beta_{\text{loc}}).
\]
On $\R^d$, an additional assumption of mild growth at infinity is required, \eg $\mathbf{v}\in (1+|x|)\cdot(L^1+L^\infty)$.
To handle unbounded data, the idea is to use renormalization, \ie to consider $\Phi(\psi)$ for suitable smooth and bounded $\Phi$.
For a review of the fundamental ideas and the last developments of the theory, we refer the reader to 
the monograph~\cite{LBL2019} by C.~Le Bris and P.L.~Lions, and the references therein.
See also the lecture notes by~L. Ambrosio and D. Trevisan~\cite{AT2017} or those of C. De Lellis~\cite{DL2007}.

\bigskip
Adding the coercive diffusion term $(-\Delta)^{\alpha/2}\psi$ in~\eqref{TDalphaCons} with $0<\alpha\leq 2$ obviously does not alter
these results. 
On the contrary, the assumptions on the transport field can even be relaxed.
For example, for $\alpha=2$ and even with a fully general second-order elliptic operator,
one can accept a field $\mathbf{v}\in L^2 + W^{1,1}$ with $(\div \mathbf{v})_-\in L^\infty$, as mentioned in \cite[\S3.2]{LBL2019}.

The local well-posedness of~\eqref{TDalphaCons} is thus classical; see \eg\cite{BJ2018}.

\begin{rmk}
If the transport term takes the conservative form, the equation is called a \textsl{conservation law};
if not, it is referred to as a general \textsl{convection}.
When the Laplace operator has variable coefficients, then the term \textsl{conservative} is preferred to describe the equation with
the operator written in divergence form~$-\partial_i(a_{ij}\partial_j)$, regardless of whether the transport part is a convection or a conservation law.
In our present case, however, the fractional power $(-\Delta)^\alpha$ is obviously a conservative operator so our use of the adjective
\textsl{conservative} concerns only the form of the advection term.
\end{rmk}

\subsection{Propagation of positivity}

The classical positivity result for $\alpha=2$ can be generalized for fractional diffusions.
\begin{prop}
If $\psi$ is a solution to~\eqref{TDalphaCons}, stemming from $\psi_0\geq0$, then $\psi(s)\geq0$ for any $s\geq0$.
\end{prop}
\begin{proof}
Let us sketch the argument first.
If a solution of~\eqref{TDalphaCons} is smooth and positive,
then at a first contact point with zero, say $(s_0,x_0)$, it reaches a global minimum. One thus has $\psi(s_0,x_0)=0$ and $\nabla\psi(s_0,x_0)=\mathbf{0}$,
and therefore:
\[
\div\left( \mathbf{v} \psi \right) = (\div \mathbf{v}) \psi(s_0,x_0) + (\mathbf{v}\cdot\nabla) \psi(s_0,x_0) =0.
\]
Moreover, for $0<\alpha<2$, by~\eqref{kernelPseudoLaplace}, there exists a positive kernel $K_{d,\alpha}$ such that:
\[
(-\Delta)^{\alpha/2} \psi (s_0,x_0) =  - \int_{\R^d} \left(\psi(s_0,y)-\psi(s_0,x_0) \right)  K_{d,\alpha}(y-x_0) dy \leq 0
\]
and the inequality is strict, unless $\psi(s_0,\cdot) \equiv0$.
The equation ensures that $\partial_s \psi (s_0,x_0)\geq 0$ and, in particular, the solution remains positive forever.
To make the proof fully rigorous, one proceeds \eg as in~\cite[\S2.1]{ISV2016}: for given $T,R>0$ and $\psi_0>0$, one considers
the approximation $\psi_R$ where the kernel~$K_{d,\alpha}$ is restricted to $\ball{0}{R}$ and
\[
s_0 = \inf \left\{ s\in(0,T) \,;\, \exists x_0 \in \ball{0}{R}, \quad \psi_R(s,x_0) = 0 \right\}.
\]
By compactness, $s_0$ is attained and $s_0>0$. As $\psi_R(s,\cdot)$ is not identically zero,
the previous computation ensures that $\partial_s \psi_R(s_0,x_0)>0$
and thus $\psi_R$ had to be negative in the neighborhood of $x_0$ a short time before $s_0$, which is contradictory.
For a general $\psi_0\geq0$, the data can be approximated by  a strictly positive mollification, whose strict positivity propagates downstream.
Passing to the limit at a later time $s>0$ ensures therefore that $\psi(s)\geq0$.
\cqfd\end{proof}

\subsection{Propagation of the $L^1$ norm and conservation of momentum}

The simple structure of~\eqref{TDalphaCons} inherited from the underlying conservation law plays in our favor.
\begin{prop}
Let $\psi$ be a solution  to~\eqref{TDalphaCons}. Then 
\begin{gather} \label{propL1Norm}
\|\psi(s,\cdot)\|_{L^1}\leq \|\psi_0\|_{L^1} \\ \intertext{and}
\label{consMomentum}
\int_{\R^d} \psi(s,x) dx = \int_{\R^d}\psi_0(x)dx.
\end{gather}
\end{prop}
\begin{proof}
For the first statement, let us decompose $\psi_0=\psi_0^+-\psi_0^-$ where both $\psi_0^+$ and $\psi_0^-$ are positive  and have disjoint supports. Let $\psi^+$ and $\psi^-$ be the solutions to the equation with initial data $\psi^+_0$ and $\psi^-_0$ correspondingly. Then, by linearity,
$\psi(s)=\psi^+(s)-\psi^-(s)$ and therefore
\[
\|\psi(s,\cdot)\|_{L^1}\leq \|\psi^+(s,\cdot)\|_{L^1}+\|\psi^-(s,\cdot)\|_{L^1}.
\]
Equation~\eqref{TDalphaCons} and an integration by part ensure that:
\begin{equation*}
\frac{d}{ds} \int_{\R^d} \psi^\pm (s,x) dx = -\int_{\R^d} (-\Delta)^{\alpha/2}\psi^\pm -\int_{\R^d}\div(\mathbf{v}\psi^\pm)dx=0.
\end{equation*}
As the propagation of positivity yields that $\psi^\pm\geq 0$,
one gets $\|\psi^\pm(s,\cdot)\|_{L^1}= \|\psi^\pm_0\|_{L^1}$ and finally $\|\psi(s,\cdot)\|_{L^1}\leq  \|\psi^+_0\|_{L^1}+ \|\psi^-_0\|_{L^1}= \|\psi_0\|_{L^1}$,
hence~\eqref{propL1Norm}. The identity~\eqref{consMomentum} is immediate.
\cqfd\end{proof}

\begin{rmk}
Note that, because of the diffusion, the functions
$\psi^\pm$ of the previous proof will not coincide, in general, with the positive and negative parts $\psi_\pm$ of $\psi$.
\end{rmk}


\subsection{Estimate of the $L^p$ norm}

For $h<d/2$ let us introduce the constant in the Sobolev embedding $\dot{H}^{h}(\R^d) \subset L^{2d/(d-2h)}$ (see~\eg\cite{T1979}):
\begin{equation}\label{sobolevEmbedding}
S_{h}(\R^d)^{-1} = \sup \left\{
\norme[L^{2d/(d-2h)}]{f}^2
\,;\,  f\in \dot{H}^{h}(\R^d), \enspace  \displaystyle \int_{\R^d} |(-\Delta)^{h/2} f|^2  = 1
\right\} >0.
\end{equation}
The idea is to relax the uniform control given by the maximum principle for~\eqref{TDalphaCons} into a weaker
one in the scale of Lebesgue spaces.

\begin{prop}\label{propagationLp}
For any $\alpha\in(0,2]$ and any dimension $d>\alpha$, if the transport field satisfies
\begin{equation}\label{smallnessAssumptiom}
(p-1) \norme[L^\infty_t L^{d/\alpha}_x]{(\div\mathbf{v})_-}  \leq S_{\alpha/2}(\R^d)
\end{equation}
for some $p\geq2$ (eventually restricted to $p=2^n$ with $n\in\N$ if $\alpha<2$),
then any solution of~\eqref{TDalphaCons} satisfies
\begin{equation}\label{estimPropagationLp1}
\norme[L^p]{\psi(s)}^p + S_{\alpha/2}(\R^d) \int_0^s \norme[L^\sigma]{ \psi(\tau) }^p d\tau  \leq \norme[L^p]{\psi_0}^p
\qquad\text{with}\qquad
\sigma = \frac{dp}{d-\alpha}
\end{equation}
and in particular
\begin{equation}\label{estimPropagationLp2}
\forall q\in[1,p],\qquad
\forall s\geq0,\qquad
\norme[L^q]{\psi(s)} \leq \norme[L^q]{\psi_0}
\end{equation}
for any $\psi_0\in L^1\cap L^p$.
In particular, when $\div(\mathbf{v})\geq0$, the estimate \eqref{estimPropagationLp2} holds for $1\leq q\leq \infty$.
\end{prop}

\begin{rmk}
The following proof also establishes that all solutions of~\eqref{TDalphaCons} satisfy:
\begin{equation}
\norme[L^p]{\psi(s)} \leq \norme[L^p]{\psi_0} e^{t (1-\frac{1}{p}) \norme[L^\infty_{t,x}]{(\div \mathbf{v})_-} }
\end{equation}
regardless of the diffusion term and independently of~\eqref{smallnessAssumptiom}.
For what follows, we are however interested in getting a better (\ie non-increasing) control of the $L^p$-norm
as given by~\eqref{estimPropagationLp1}-\eqref{estimPropagationLp2}.
\end{rmk}

\begin{proof}
Using $p|\psi|^{p-2}\psi$  as a multiplier for the equation leads to:
\[
\frac{d}{ds} \int_{\R^d} |\psi(s,x)|^p dx
+ p \int_{\R^d}  |\psi|^{p-2}\psi \cdot (-\Delta)^{\alpha/2}\psi 
=
- p \int_{\R^d}  \div\left( \mathbf{v} \psi \right) |\psi|^{p-2}\psi.
\]
For the integral on the right-hand side, an integration by  part gives:
\begin{align*}
\int_{\R^d}  \div\left( \mathbf{v} \psi \right) |\psi|^{p-2}\psi
& = - (p-1)\int_{\R^d}    |\psi|^{p-2}\psi (\mathbf{v}\cdot\nabla) \psi \\
& = - (p-1)\int_{\R^d}    \div\left( \mathbf{v} \psi \right) |\psi|^{p-2}\psi + (p-1)  \int_{\R^d}|\psi|^p \div\mathbf{v}.
\end{align*}
One thus has this identity:
\begin{equation}\label{identityRhsLpEstim}
\int_{\R^d}  \div\left( \mathbf{v} \psi \right) |\psi|^{p-2}\psi = \left(1-\frac{1}{p}\right)  \int_{\R^d}|\psi|^p \div\mathbf{v}
\end{equation}
and thus
\begin{equation}\label{startingLpEstimate}
\frac{d}{ds} \int_{\R^d} |\psi(s,x)|^p dx
+ p \int_{\R^d}  |\psi|^{p-2}\psi \cdot (-\Delta)^{\alpha/2}\psi 
= - (p-1) \int_{\R^d} |\psi|^p \div\mathbf{v}.
\end{equation}
For the integral on the left-hand side of~\eqref{startingLpEstimate} and when $\alpha=2$, the following identity holds:
\begin{subequations}
\begin{equation}\label{heatLp}
p \int_{\R^d}  |\psi|^{p-2}\psi \cdot (-\Delta)\psi  = p(p-1) \int_{\R^d}  |\psi|^{p-2} |\nabla \psi|^2 
= 4\left(1-\frac{1}{p}\right)\int_{\R^d} |\nabla (|\psi|^{p/2})|^2 \geq 0.
\end{equation}
For $0<\alpha<2$, one needs to replace~\eqref{heatLp} because the Leibniz formula is no longer valid;
instead, one follows the ideas of~\cite{CC2004}.
The key is the point-wise inequality \cite[Prop. 2.3]{CC2004}:
\[
2\psi \cdot (-\Delta)^{\alpha/2}\psi\geq (-\Delta)^{\alpha/2}(|\psi|^2)
\]
which follows immediately from the kernel representation~\eqref{kernelPseudoLaplace} of $(-\Delta)^{\alpha/2}$. Applied recursively $n-1$ times
when $p=2^n$ and $n\geq1$ is an integer, it provides for $1\leq k\leq n-1$ (or without intermediary $k$ if $n=1$):
\begin{equation}\label{fractionalHeatLp}
p\int_{\R^d}  |\psi|^{p-2}\psi \cdot (-\Delta)^{\alpha/2}\psi 
\geq 
\frac{p}{2^k}\int_{\R^d}  |\psi|^{p-2^k} (-\Delta)^{\alpha/2}(|\psi|^{2^k})
\geq
2\int_{\R^d} |(-\Delta)^{\alpha/4}(\psi^{p/2})|^2.
\end{equation}
\end{subequations}
Compared to \cite[Lemma 2.4]{CC2004}, the inequality~\eqref{fractionalHeatLp} is improved by a factor of 2.
Overall, for $p\geq2$ (restricted to exact powers of $2$ if $0<\alpha<2$), the evolution of the $L^p$-norm of  smooth solutions of~\eqref{TDalphaCons} obeys the following inequality:
\begin{equation}\label{groundLpEstimate}
\frac{d}{ds} \norme[L^p]{\psi}^p + 2\int_{\R^d} |(-\Delta)^{\alpha/4} (\psi^{p/2})|^2  \leq 
- (p-1) \int_{\R^d} |\psi|^p \div\mathbf{v}.
\end{equation}
Obviously, only the focusing zones (\ie regions where $\div\mathbf{v}<0$) of the transport field can contribute to an increase of the $L^p$ norm;
the other just tends to spread $\psi$ out. Using the notation~\eqref{PosNegNotation} for the negative part, one thus has the following estimate:
\begin{equation}\label{groundLpEstimateSigned}
\frac{d}{ds} \norme[L^p]{\psi}^p + 2\int_{\R^d} |(-\Delta)^{\alpha/4} (\psi^{p/2})|^2 
\leq
(p-1) \int_{\R^d} |\psi|^p (\div\mathbf{v})_-.
\end{equation}
In dimension $d\geq2$ and for $0<\alpha<2$, one uses the Sobolev embedding \eqref{sobolevEmbedding}.
For $\sigma = dp/(d-\alpha)$, one thus has:
\[
\norme[L^\sigma]{\psi}^p  = \norme[L^{2d/(d-\alpha)}]{\psi^{p/2}}^2
\leq  \frac{1}{S_{\alpha/2}(\R^d)}  \int_{\R^d} |(-\Delta)^{\alpha/4} \psi^{p/2}|^2 .
\]
The estimate~\eqref{groundLpEstimateSigned} becomes:
\begin{equation}\label{groundLpEstimateSignedSimplified}
\frac{d}{ds} \norme[L^p]{\psi}^p
+ 2 S_{\alpha/2}(\R^d) \cdot \norme[L^\sigma]{\psi}^p
\leq
(p-1) \int_{\R^d} |\psi|^p (\div\mathbf{v})_-.
\end{equation}
Finally, as the conjugate exponent of  $q=d/\alpha>1$ satisfies $p q' = \sigma$, one splits the right-hand side in the following way:
\[
\int_{\R^d} |\psi|^p (\div\mathbf{v})_- \leq
\norme[L^\sigma]{\psi}^p \norme[L^{d/\alpha}]{(\div\mathbf{v})_-}.
\]
Thanks to the smallness assumption~\eqref{smallnessAssumptiom},
it is then possible to bootstrap the Lebesgue norm into the left-hand side.
In that case, \eqref{groundLpEstimateSignedSimplified} ensures that $\frac{d}{ds}\norme[L^p]{\psi} + S_{\alpha/2}(\R^d) \cdot \norme[L^\sigma]{\psi}^p\leq 0$,
which gives~\eqref{estimPropagationLp1}.
One can then interpolate with~\eqref{propL1Norm} to control all $L^q$ norms for $1\leq q\leq p$.
\cqfd\end{proof}

\begin{rmk}\label{LargerV}
If $C_{\alpha,p}(\mathbf{v})=2S_{\alpha/2}(\R^d)-(p-1) \norme[L^\infty_t L^{d/\alpha}_x]{(\div\mathbf{v})_-}>0$ then~\eqref{estimPropagationLp1} still holds, but with the constant $S_{\alpha/2}(\R^d)$ replaced by $C_{\alpha,p}(\mathbf{v})$, which is not uniform in $\mathbf{v}$ anymore.
\end{rmk}


\begin{rmk}
An improved version of~\eqref{fractionalHeatLp} valid for average-free functions is found in \cite{CGHV2014} or \cite[Prop.~2.4]{CTV2015}:
\[
\int_{\R^d}  |\psi|^{p-2}\psi \cdot (-\Delta)^{\alpha/2}\psi  \geq \frac{1}{p} \norme[L^2]{(-\Delta)^{\alpha/2}(\psi^{p/2})}^2
+ C \norme[L^p]{\psi}^p.
\]
However, in our case, using a Sobolev embedding for $\psi^{p/2}$ provides some additional integrability
and in particular a control of the $L^\sigma$-norm with $\sigma>p$. This gain will be crucial in what follows.
It also allows us to put a restriction on the $L^{d/\alpha}$-norm of $(\div\mathbf{v})_-$, instead of requiring smallness in~$L^\infty$.
\end{rmk}

Note that on $\T^d$, the Sobolev embedding $\dot{H}^{h}(\T^d) \subset L^{2d/(d-2h)}$~\eqref{sobolevEmbedding} is only valid for average-free functions. However,~$\psi^{p/2}$ is not average-free in general (\ie $p\neq2$), even if $\psi$ is so.
For the periodic case, one will use the following simpler result, whose proof is also contained above.
\begin{prop}\label{periodicCase}
If $\psi$ is an average-free solution of~\eqref{TDalphaCons} on $\T^d$ with $\alpha\in(0,2]$ and $d>\alpha$ and
\[
\norme[L^\infty_t L^{d/\alpha}_x]{(\div\mathbf{v})_-}  \leq S_{\alpha/2}(\T^d),
\]
then
\[
\norme[L^2]{\psi(s)}^2 + S_{\alpha/2}(\T^d) \int_0^s \norme[L^\sigma]{ \psi(\tau) }^2 d\tau  \leq \norme[L^2]{\psi_0}^2
\qquad\text{with}\qquad
\sigma = \frac{2d}{d-\alpha}\cdotp
\]
\end{prop}

\section{Propagation of the atom property by the dual conservation law}\label{par:propAtom}

As long as the advection field has mildly convergent characteristics (expressed precisely by~\eqref{smallnessAssumption2}),
the weak maximum principle implies that the (non-local) diffusion 
propagates the properties of atoms. It is possible to trade a slow increase in each atomic radius to
gain some decay in amplitude.

\subsection{Local propagation}

\begin{prop}\label{PropagationAtom}
Let us assume that $1\leq \alpha\leq 2$ and $d>\alpha$ and that the velocity field $\mathbf{v} \in \mathrm{BMO}$ satisfies
\begin{equation}\label{smallnessAssumption2}
(p-1) \norme[L^\infty_t L^{d/\alpha}_x]{(\div\mathbf{v})_-}  \leq S_{\alpha/2}(\R^d)
\end{equation}
for some $p\geq2$ (eventually restricted to $p=2^n$ with $n\in\N$ if $\alpha<2$) such that
\[
p > \frac{d}{d-(\alpha-\omega)} \qquad\text{with}\qquad
0<\omega<\alpha\wedge 1.
\]
Then there exist constants $\delta, K$ and $\gamma$, depending only on $d$, $p$, $\alpha$
and $\|\mathbf{v}\|_{\mathrm{BMO}}$, such that for all $r\in (0,1]$, the following implication holds:
\begin{equation}
\psi_0\in \atom^p_{r}
\qquad\Longrightarrow\qquad
\forall s\in[0,\gamma r^\alpha],\quad
\psi(s,\cdot)\in \left(1-\frac{\delta s}{r^{\alpha}}\right) \atom^p_{(r^\alpha+Ks)^{1/\alpha}}
\end{equation}
where $\psi$ denotes the solution of the Cauchy problem~\eqref{TDalphaCons}.
The constant $A$, which is implicit in the definition of $\atom^p_{r}$, has to be chosen large enough.
The admissible threshold for $A$, which also depends only on $d$, $p$, $\alpha$
and $\|\mathbf{v}\|_{\mathrm{BMO}}$, is specified in remark~\ref{choiceOfA}.
\end{prop}

\begin{rmk}\label{choiceOfp}
The proposition holds with $p=2$ if $d>2(\alpha-\omega)$, which is always possible if one chooses~$\omega$
such that $\alpha-1 < \omega<1$ when $\alpha<2$ (and $\omega>1/2$ when $\alpha=2$ and $d\geq3$);
in this case~\eqref{smallnessAssumption2} is also the least restrictive.
Thanks to proposition~\ref{periodicCase}, the result then also holds, mutatis mutandis, on $\T^d$.
\end{rmk}

\begin{proof}
The proof of proposition~\ref{PropagationAtom} is inspired by~\cite{KN2009} and \cite{D2010},
though the fractional derivative requires some additional care. 
Thanks to~\eqref{consMomentum}, the zero-average property of atoms is obviously propagated
by~\eqref{TDalphaCons}.

\medskip
Let $x(s)$ be the solution to the following ODE, which tracks the average flow on a ball of size $r$. It is obviously well defined for $\mathbf{v}\in L^1_{\text{loc}}(\R_+\times\R^d)$: 
\begin{equation}\label{trackCenter}
\begin{cases}
x^\prime(s)=\bar{\mathbf{v}}_{\ball{x(s)}{r}}\\
x(0) = x_0
\end{cases}
\qquad\text{where}\qquad
\bar{\mathbf{v}}_{\ball{x}{r}}(s) = \frac{1}{|\ball{x}{r}|}\int_{\ball{x}{r}} \mathbf{v}(s,y) dy.
\end{equation}

\paragraph{Step 1. Strict decay of the $L^1$-norm.}
One introduces
$S=\psi(s)^{-1}(\{0\})$ and  $D_\pm = \operatorname{supp}\psi_\pm(s)$.
Arguing as in~\cite[\S4]{KN2009} and taking advantage of the conservative form of~\eqref{TDalphaCons}:
\[
\frac{d}{ds}\norme[L^1]{\psi(s)} = 
-\int_{D_\pm} \frac{\psi}{|\psi|} (-\Delta)^{\alpha/2}\psi + \int_{S} |(-\Delta)^{\alpha/2}\psi|.
\]
The kernel formula~\eqref{kernelPseudoLaplace} allows us to improve~\eqref{propL1Norm} and gives
\begin{equation}\label{propL1NormImproved}
\frac{d}{ds}\norme[L^1]{\psi(s)} \leq -\frac{c_{d,\alpha}}{2}
\iint_{D_\pm^2} \left(\frac{\psi(y)}{|\psi(y)|}-\frac{\psi(x)}{|\psi(x)|}\right) \frac{\psi(y)-\psi(x)}{|y-x|^{d+\alpha}} dydx
\end{equation}
because (if $S$ is a set of strictly positive measure)
\[
\int_{S} |(-\Delta)^{\alpha/2}\psi (y)| dy
\leq c_{d,\alpha}\int_{x\in D_\pm} \int_{y\in S} \frac{|\psi(x)|}{|y-x|^{d+\alpha}} dydx.
\]
The right-hand side of \eqref{propL1NormImproved} is negative:
\[
\frac{d}{ds}\norme[L^1]{\psi(s)} \leq -c_{d,\alpha}\left[
\int_{D_+} \left(\int_{D_-} \frac{dy}{|x-y|^{d+\alpha}}\right) \psi_+(s,x) dx
+\int_{D_-} \left(\int_{D_+} \frac{dy}{|x-y|^{d+\alpha}}\right) \psi_-(s,x) dx
\right].
\]
Obviously, the domains of integration $D_\pm$ can be reduced to $D_\pm\cap \ball{x(s)}{100 r}$.
Provided that most of the $L^1$-mass of $\psi$ is localized in $\ball{x(s)}{100 r}$, which, as explained in~\cite{KN2009}, is ensured by the 3rd part of the proof, it ends up giving:
\[
\frac{d}{ds}\norme[L^1]{\psi(s)} \lesssim - r^{-\alpha}
\]
\ie for $\delta$ and $\gamma>0$ small enough
\begin{equation}
\forall s\in[0,\gamma r^\alpha],\qquad
\norme[L^1]{\psi(s)} \leq 1-\frac{\delta s}{r^{\alpha}}\cdotp
\end{equation}

\paragraph{Step 2. Strict decay of the $L^p$ norm.}
We have already proven that, under the smallness assumption~\eqref{smallnessAssumption2},
the right-hand side of~\eqref{groundLpEstimateSignedSimplified} can be resorbed within the elliptic term, \ie
\[
\frac{d}{ds} \norme[L^p]{\psi}^p
\leq -  S_{\alpha/2}(\R^d) \cdot \norme[L^\sigma]{\psi}^p <0.
\]
Next, as $\sigma=\frac{dp}{d-\alpha}>p$, one can use an elegant idea of \cite[p.~517]{CC2004}, which is to combine the interpolation inequality
$\norme[L^p]{\psi} \leq \norme[L^1]{\psi}^{1-\theta} \norme[L^\sigma]{\psi}^\theta$ for $\theta = \frac{(p-1)d}{(p-1)d+\alpha}$
with the propagation of the $L^1$-norm~\eqref{propL1Norm}. As~$\psi_0$ is an atom~$\atom^p_{r}$, it ensures that:
\begin{equation}
\frac{d}{ds} \norme[L^p]{\psi}^p
\leq
- S_{\alpha/2}(\R^d) \cdot \norme[L^1]{\psi_0}^{-\frac{\alpha p}{(p-1)d}}
\norme[L^p]{\psi}^{p/\theta}
\leq
-S_{\alpha/2}(\R^d) \cdot
\norme[L^p]{\psi}^{p/\theta}.
\end{equation}
This is a Riccati-type ODE that can be solved explicitly:
\[
\norme[L^p]{\psi(s)}^p \leq 
\left(
\norme[L^p]{\psi_0}^{-\frac{\alpha p}{(p-1)d}} + {\textstyle \frac{\alpha S_{\alpha/2}(\R^d)}{(p-1)d}} \cdot s
\right)^{-\frac{(p-1)d}{\alpha}}.
\]
Using the atom property $\norme[L^p]{\psi_0} \leq A r^{-d(1-\frac{1}{p})}$ and
rearranging the terms, one gets:
\[
\norme[L^p]{\psi(s)} \leq 
A r^{-d(1-\frac{1}{p})} \left(
1 + {\textstyle \frac{\alpha S_{\alpha/2}(\R^d)}{(p-1)d}}\cdot A^{\frac{\alpha p}{(p-1)d}} r^{-\alpha}  s
\right)^{-\frac{d}{\alpha}(1-\frac{1}{p})}
\]
\ie $\norme[L^p]{\psi(s)} \leq A r(s)^{-d(1-\frac{1}{p})}$ with
\begin{equation}
r(s) = \left( r^{\alpha}  + C_{d,p,\alpha} A^\mu  s \right)^{1/\alpha}
\quad\text{and}\quad
\mu = \frac{\alpha}{d(1-\frac{1}{p})} \cdotp
\end{equation}
One chooses $\delta>0$ small enough, then
\begin{equation}\label{largeA1}
0<K<C_{d,p,\alpha} A^\mu - \frac{2\delta}{\frac{d}{\alpha}\left(1-\frac{1}{p}\right)}\cdotp
\end{equation}
Thanks to the reversed Bernoulli inequality $(1-x)^{-1/\beta}\leq 1+2x/\beta$ for $\beta>1$ and $x\in[0,1/2]$, this choice ensures that
\[
\forall t\in [0,\gamma],\qquad
\frac{1+ C_{d,p,\alpha} A^\mu t}{1+ K t}
\geq  1  + \frac{2\delta t}{\frac{d}{\alpha}{(1-\frac{1}{p})}}
\geq (1-\delta t)^{-1/[\frac{d}{\alpha}(1-\frac{1}{p})]}
\]
with $\gamma = \frac{\frac{d}{\alpha}(1-\frac{1}{p})}{2\delta}\left(
\frac{ C_{d,p,\alpha} A^\mu - 2\delta / [ \frac{d}{\alpha}(1-\frac{1}{p}) ]}{K} - 1
\right)$
and thus, after substituting $t=s/r^\alpha$:
\begin{equation}
\norme[L^p]{\psi(s)}
\leq A \left( 1-\frac{\delta s}{r^\alpha} \right) (r^\alpha+K s)^{-\frac{d}{\alpha}(1-\frac{1}{p})}
\end{equation}
for $s\in[0,\gamma r^\alpha]$.

\paragraph{Step 3. Propagation of the concentration.}
With $x(s)$ defined by~\eqref{trackCenter}, one considers
\begin{equation}
\chi(s) =  \int_{\R^d}\psi(s,x) \Omega(x-x(s)) dx .
\end{equation}
Using the equation~\eqref{TDalphaCons} and the fact that $(-\Delta)^{\alpha/2}$ is self-adjoint, the derivative of $\chi$ satisfies:
\begin{equation}\label{eqHprime}
\chi'(s)=\int_{\R^d} (\mathbf{v}-\bar{\mathbf{v}}_{\ball{x(s)}{r}} ) \cdot \nabla \Omega(x-x(s)) \psi(s,x)dx 
-\int_{\R^d} \psi(s,x) \cdot (-\Delta)^{\alpha/2}\Omega(x-x(s))  dx.
\end{equation}
Let us collect obvious estimates for the derivatives of $\Omega$:
\begin{subequations}\label{omegaEstimates}
\begin{gather}
| \nabla\Omega (z)| \lesssim |z|^{-(1-\omega)} \cdot \mathbf{1}_{\ball{0}{2}}(z),\\
\label{OmegaSecondEstimate}
| (-\Delta)^{\alpha/2}\Omega(z) | \lesssim 
 |z|^{-(\alpha-\omega)_+} \cdot \mathbf{1}_{\ball{0}{2}}(z) +
 |z|^{-2-\alpha} \cdotp  \mathbf{1}_{\ball{0}{2}^c}(z).
\end{gather}
\end{subequations}
They follow easily from the scaling properties of the Fourier transform (and thus of $(-\Delta)^{\alpha/2}$) and from the kernel representation~\eqref{kernelPseudoLaplace}. Recall that we assume $\omega<\min\{\alpha,1\}$ throughout the proof.

\subparagraph{3a. Transport term in $\chi'$.}
Let us introduce $E_k(s):=\{x\in \R^d: |x-x(s)|\in [2^{k-1}r,2^kr)\}$ to estimate
\[
I_1 =\left| \int_{\T^d} (\mathbf{v}-\bar{\mathbf{v}}_{\ball{x(s)}{r}})\cdot \nabla\Omega(x-x(s))\psi(s,x) dx \right|.
\]
One has $I_1 \lesssim J_0 + \sum\limits_{k=1}^\infty J_k$  with
\[
J_0 = \int_{\ball{x(s)}{r}}|\mathbf{v}-\bar{\mathbf{v}}_{\ball{x(s)}{r}}| |x-x(s)|^{-(1-\omega)} |\psi| 
\]
and
\[
J_k = \left( \int_{E_k(s)}|\mathbf{v}-\bar{\mathbf{v}}_{\ball{x(s)}{r}}| |\psi | \right)r^{-(1-\omega)} 2^{-k(1-\omega)}.
\]
For $J_0$, we use the H\"older inequality with $a^{-1}+b^{-1}+c^{-1}=1$ and the BMO property
\begin{align*}
J_0 &\leq \|\mathbf{v}-\bar{\mathbf{v}}_{\ball{x(s)}{r}}\|_{L^a(\ball{x(s)}{r})} \|\psi_0\|_{L^b} \||x-x(s)|^{-(1-\omega)}\|_{L^c(\ball{x(s)}{r})} \\
& \lesssim \norme[\BMO]{\mathbf{v}} \cdot r^{d/a} \cdot A^{b_\ast} r^{-d(1-\frac{1}{b})} \cdot r^{\frac{d}{c}-(1-\omega)}\\
& \lesssim r^{-(1-\omega)} A^{b_\ast} \norme[\BMO]{\mathbf{v}}
\end{align*}
with
\begin{equation}
b_\ast = \frac{1-\frac{1}{b}}{1-\frac{1}{p}} = \frac{p'}{b'}\cdotp
\end{equation}
Here, one should comment on the choice of the powers $a,b,c$.
Obviously, we have to take $c<d/(1-\omega)$ for local integrability reasons.
Second, as we used the decay of the $L^b$ norm of $\psi$ given by proposition~\ref{propagationLp}
followed by proposition~\ref{AtomL1LpBound} on $\psi_0$, one needs $p\geq b>d/(d-(1-\omega))$.
Since $a$ can be chosen arbitrary large, it is always possible to find a proper triplet $(a,b,c)$ as soon as
\begin{equation}\label{restrictP}
p > \frac{d}{d-(1-\omega)}\cdotp
\end{equation}
For $J_k$ with $k\geq 1$, we apply the H\"older inequality with a pair of conjugate powers $q_1$ and $q'_1$,
with $q_1>d/(1-\omega)$. Thanks to~\eqref{restrictP},  one thus has $q'_1<\frac{d}{d-(1-\omega)}<p$,
which ensures again that we have propagation of the $L^{q'_1}$ norm of $\psi$ and
that proposition~\ref{AtomL1LpBound} may be used liberally on $\psi_0$.
One also uses that for $\BMO$ functions, the averages of adjacent dyadic balls are comparable
and that $\norme[L^1(E_k)]{\psi} \leq 2^{kd/q_1} \norme[L^{q'_1}(E_k)]{\psi}$
uniformly for $r\in (0,1]$. One thus gets:
\begin{align*}
 \int_{E_k(s)}|\mathbf{v}-\bar{\mathbf{v}}_{\ball{x(s)}{r}}| |\psi | 
 &\leq
 \int_{E_k(s)}|\mathbf{v}-\bar{\mathbf{v}}_{\ball{x(s)}{2^kr}}| |\psi |
 + \int_{E_k(s)}|\bar{\mathbf{v}}_{\ball{x(s)}{2^kr}}-\bar{\mathbf{v}}_{\ball{x(s)}{r}}| |\psi | \\
& \lesssim
\norme[ L^{q_1}({\ball{x(s)}{2^k r}}) ]{ \mathbf{v}-\bar{\mathbf{v}}_{\ball{x(s)}{2^k r}} }
\norme[L^{q'_1}]{\psi_0}
 + k  \norme[\BMO]{\mathbf{v}}  \norme[L^1(E_k)]{\psi} \\[2pt]
&\lesssim
(1+k) 2^{kd/q_1} A^{p'/q_1}
\norme[\BMO]{\mathbf{v}} .
\end{align*}
As we choose $d/q_1<1-\omega$, the geometric series in $k\geq1$ is convergent and $p'/q_1<b_\ast$,  and thus:
\begin{equation}\label{estimateI1}
I_1\lesssim r^{-(1-\omega)}  A^{b_\ast} 
\norme[\BMO]{\mathbf{v}}.
\end{equation}
Let us observe that, due to the admissible range for $b$, the value of $b_\ast$ can be chosen arbitrarily within the interval
\[
\frac{1-\omega}{d(1-\frac{1}{p})} < b_\ast \leq 1.
\]
In the next part of this proof, we will chose $b_\ast$ to be as close as possible to the lowest bound.

\begin{rmk}
As $\operatorname{supp} \nabla \Omega \subset \ball{0}{2}$ the series of $J_k$ terms is limited to $k\lesssim |\log r|$.
However, this upper bound becomes arbitrarily large when $r\to0$. Our previous estimate is uniform for $r\in(0,1]$.
\end{rmk}

\subparagraph{3b. Nonlocal viscous term in $\chi'$.}
Let us now consider the second term of~\eqref{eqHprime}:
\[
I_2=\biggl|\int_{\R^d} \psi(s,x) \cdot (-\Delta)^{\alpha/2} \Omega(x-x(s)) dx\biggr|.
\]
Recall that we assume $\alpha>\omega$.
Thanks to~\eqref{OmegaSecondEstimate}, for any $0< \rho\leq r<1$, one has:
\[
I_2 \lesssim
\int_{\ball{x(s)}{\rho}}|x-x(s)|^{-(\alpha-\omega)} |\psi(s,x)| dx
+ \rho^{-(\alpha-\omega)} \norme[L^1]{\psi_0}.
\]
We apply the H\"older inequality with another pair of conjugate powers $q_2$ and $q'_2$, with $1\leq q_2<\frac{d}{\alpha-\omega}$,
which is always possible.
One also needs $q_2' \leq p$ to ensure the propagation of the $L^{q_2'}$ norm by proposition~\ref{propagationLp},
\ie $\frac{1}{q_2}\leq 1-\frac{1}{p}$. Such a choice is possible if
\begin{equation}\label{restrictP2}
p > \frac{d}{d-(\alpha-\omega)}\cdotp
\end{equation}
As $\alpha\geq1$, this restriction on $p$ supersedes~\eqref{restrictP}.
One gets
\[
I_2 \leq 
\norme[ L^{q_2}(\ball{x(s)}{\rho}) ]{ |x-x(s)|^{-(\alpha-\omega)}  }
\norme[ L^{q_2'} ]{\psi_0}
+ \rho^{-(\alpha-\omega)} \norme[L^1]{\psi_0} 
\]
\ie
\[
I_2 \lesssim 
\rho^{\frac{d}{q_2} - (\alpha-\omega)} A^{p'/q_2} r^{-d/q_2}
+ \rho^{-(\alpha-\omega)}.
\]
The optimal choice for $\rho$ is given by $\rho = r A^{-p'/d}$,
which belongs indeed to $(0,r]$ as $A\gg1$.  Substituting this value in the previous estimate of $I_2$ gives:
\begin{equation}\label{estimateI2}
I_2 \lesssim r^{-(\alpha-\omega)} A^{\mu_\ast}
\qquad\text{with}\qquad \mu_\ast = \frac{\alpha-\omega}{d(1-\frac{1}{p})} \cdotp
\end{equation}

\subparagraph{3c. Conclusion.}

Putting together~\eqref{eqHprime} with \eqref{estimateI1} and \eqref{estimateI2}, one gets:
\[
|\chi'(s)|\leq I_1+I_2 \lesssim 
r^{-(1-\omega)} 
A^{b_\ast}
\norme[\BMO]{\mathbf{v}} 
+ r^{-(\alpha-\omega)} A^{\mu_\ast}.
\]
After integration and considering that $\chi(0)\leq r^{\omega}$ and $1\leq r^{-1}\leq r^{-\alpha}$ because $\alpha\geq1$:
\[
\chi(s) \leq r^{\omega}\left( 1 + C_{d,p,\alpha}'
\left[
A^{b_\ast} \norme[\BMO]{\mathbf{v}} 
+ A^{\mu_\ast}
\right] \frac{s}{r^{\alpha}}  \right).
\]
Provided $A$ is large enough, one may amend the previous choice~\eqref{largeA1} of $\delta$ and $K$ to ensure that
\begin{equation}\label{largeA2}
K \geq \frac{\alpha}{\omega} \left\{ \delta + C_{d,p,\alpha}' \left[ A^{b_\ast}
\norme[\BMO]{\mathbf{v}} 
+ A^{\mu_\ast}
\right]
\right\},
\end{equation}
which, in turn, ensures that
\[
\forall t\in[0,\gamma],\qquad
1 + C_{d,p,\alpha}'
\left[
A^{b_\ast}
\norme[\BMO]{\mathbf{v}} 
+ A^{\mu_\ast}
\right] t 
\leq 
(1-\delta t)(1+Kt)^{\omega/\alpha}
\]
\ie, with $t=s/r^\alpha$:
\begin{equation}
\forall s\in[0,\gamma r^{\alpha}],\qquad
\chi(s) \leq \left(1-\frac{\delta s}{r^\alpha}\right) (r^\alpha+K s)^{\omega/\alpha}.
\end{equation}
This concludes the proof of $\psi(s)\in \left(1-\frac{\delta s}{r^{\alpha}}\right) \atom^p_{(r^\alpha+Ks)^{1/\alpha}}$.
\cqfd\end{proof}

\begin{rmk}\label{choiceOfA}
Our conditions~\eqref{largeA1} and~\eqref{largeA2} generalize respectively the conditions 4.6 - 4.15 of~\cite{KN2009}.
Both conditions are compatible, provided $A$ is chosen large enough, because
\[
\mu > b_\ast \vee \mu_\ast.
\]
In turn, this condition is satisfied by choosing $b_\ast$ as small as possible and because $\alpha>\omega \vee (1-\omega)$.
\end{rmk}

\begin{rmk}\label{troubleWithL1}
It could be tempting to discard the $L^1$-property from the atom definition and use proposition~\ref{AtomL1LpBound2} from the appendix
to control this norm a-posteriori.
In this case, instead of~\eqref{largeA1} and~\eqref{largeA2}, one is led to choose $\delta$ and $K$ such that
\[
\frac{\alpha}{\omega} \left\{ \delta + C_{d,p,\alpha}' \left[ A^{\tilde{b}}
\norme[\BMO]{\mathbf{v}} 
+ A^{\tilde\mu}
\right]
\right\} \leq
 K<C_{d,p,\alpha} A^{\tilde\mu} - \frac{2\delta}{\frac{d}{\alpha}\left(1-\frac{1}{p}\right)}
\]
with $\tilde{b} = \frac{\omega+d(1-\frac{1}{b})}{\omega+d(1-\frac{1}{p})} \in (\tilde\mu/\alpha,1]$ and $\tilde\mu = \frac{\alpha}{\omega+d(1-\frac{1}{p})}$.
As both sides are order $A^{\tilde\mu}$, it is not clear anymore that the choice can be resolved for some
large value of $A$. This alternate path is thus a subtle deadlock. 
\end{rmk}


\subsection{Global propagation}

\begin{prop}\label{PropagationAtomGlobal}
In the conditions of Proposition \ref{PropagationAtom}, the constants $\delta$, $K$ are such that
\begin{equation}
\psi_0\in \atom^p_{r}
\qquad\Longrightarrow\qquad
\forall s>0,\qquad
\psi(s,\cdot)\in \left(\frac{r^\alpha}{r^\alpha+Ks}\right)^{\delta /K} \atom_{(r^\alpha+Ks)^{1/\alpha}}
\end{equation}
where $\psi$ denotes the solution of the Cauchy problem~\eqref{TDalphaCons}.
\end{prop}
\begin{proof} 
We keep the assumptions and notations of proposition \ref{PropagationAtom}.
Let us split the time-line in consecutive intervals $[\ell \gamma r^\alpha, (\ell+1)\gamma r^\alpha]$ with $\ell\in\N$.
For $\ell=0$, one has
\[
\forall s\in [0,\gamma r^\alpha],\qquad
1-\frac{\delta s}{r^{\alpha}} \leq \left(1+\frac{Ks}{r^\alpha}\right)^{-\delta /K} 
 = \left(\frac{r^\alpha}{r^\alpha+Ks}\right)^{\delta /K}.
\]
Let us assume that, for some integer $\ell\in\N$, one has:
\[
\psi(\ell \gamma r^\alpha,\cdot) \in
(1+K\ell \gamma)^{-\delta/K}
\atom_{r(1+K \ell \gamma)^{1/\alpha}}.
\]
Then for any $s\in [\ell \gamma r^\alpha, (\ell+1)\gamma r^\alpha]$, proposition \ref{PropagationAtom} gives:
\[
\psi(s,\cdot) \in
(1+K\ell \gamma)^{-\delta/K} \left(1+\frac{KS}{R^\alpha}\right)^{-\delta /K} 
\atom_{(R^\alpha+K S)^{1/\alpha}}
\]
with $S = s-\ell \gamma r^\alpha$ and $R=r(1+K \ell \gamma)^{1/\alpha}$.
The new radius is an exact match:
\begin{equation}\label{radiusMatch}
(R^\alpha+K S)^{1/\alpha} = (r^\alpha + Ks)^{1/\alpha}.
\end{equation}
Similarly, the amplitude satisfies:
\[
(1+K\ell \gamma) \left(1+\frac{KS}{R^\alpha}\right) = 1 + \frac{Ks}{r^\alpha} \cdotp
\]
The proposition thus follows by induction on $\ell\in\N$.
\cqfd\end{proof}

\subsection{Modifications in the case $0<\alpha<1$}

Throughout \S\ref{par:propAtom}, the assumption $\alpha\geq1$ has only been used in the third step of the proof of proposition~\ref{PropagationAtom},
\ie to ascertain the propagation of the concentration. Let us investigate in this subsection how to deal with the case $0<\alpha<1$.

\bigskip
When dealing with the super-critical case, L.~Silvestre \cite{Silv2010} assumes a higher regularity for the advection field.
We will do the same here and assume respectively that $v\in C^{1-\alpha}(\R^d)$ or $C^{1-\alpha}(\T^d)$.
The identity~\eqref{eqHprime} still holds. To deal with the transport term, one uses
\[
\left|\mathbf{v}(x)-\bar{\mathbf{v}}_{\ball{x(s)}{r}}\right| \leq \frac{1}{|\ball{x(s)}{r}|}\int_{\ball{x(s)}{r}} |\mathbf{v}(x) - \mathbf{v}(y)| dy
\leq  \norme[C^{1-\alpha}]{\mathbf{v}}
\:\rule{10pt}{1pt}\hspace{-12pt}\int_{\ball{x(s)}{r}}|x-y|^{1-\alpha} dy,
\]
which gives
\begin{equation}\label{improvedVminusVBar}
\left|\mathbf{v}(x)-\bar{\mathbf{v}}_{\ball{x(s)}{r}}\right| \leq
(|x-x(s)|+r)^{1-\alpha} \norme[C^{1-\alpha}]{\mathbf{v}}.
\end{equation}
For $J_0$, one takes $a=\infty$ and the same constraints for the exponents $b$ and $c$, thus
\[
J_0\lesssim r^{-(\alpha-\omega)}  A^{b_\ast} \norme[C^{1-\alpha}]{\mathbf{v}}.
\]
Similarly, for $J_k$, one takes $q_1=\infty$:
\[
 \int_{E_k(s)}|\mathbf{v}-\bar{\mathbf{v}}_{\ball{x(s)}{r}}| |\psi | 
 \lesssim 2^{k(1-\alpha)} r^{1-\alpha} \norme[C^{1-\alpha}]{\mathbf{v}}
\]
thus $J_k\lesssim 2^{-k(\alpha-\omega)} r^{-(\alpha-\omega)} \norme[C^{1-\alpha}]{\mathbf{v}}$
and as $\omega<\alpha$, the geometric series in $k$ is convergent. The estimate~\eqref{estimateI1} can therefore be replaced by
\begin{equation}\label{estimateI1bis}
I_1\lesssim r^{-(\alpha-\omega)}  A^{b_\ast} 
\norme[C^{1-\alpha}]{\mathbf{v}}
\qquad\text{with}\quad
b_\ast = \frac{1-\omega}{d(1-\frac{1}{p})} + \varepsilon, \quad \varepsilon>0.
\end{equation}
For the non-local viscous term, the estimate~\eqref{estimateI2} remains valid. The sole difference is that now
\[
\frac{d}{d-(1-\omega)} > \frac{d}{d-(\alpha-\omega)}
\]
and, consequently, the requirement \eqref{restrictP} trumps \eqref{restrictP2}.

\bigskip
Putting together~\eqref{eqHprime} with \eqref{estimateI1bis} and \eqref{estimateI2}, one gets:
\[
|\chi'(s)| \lesssim r^{-(\alpha-\omega)} \left( A^{b_\ast} 
\norme[C^{1-\alpha}]{\mathbf{v}}  + A^{\mu_\ast}\right)
\]
and
\[
\chi(s) \leq r^{\omega}\left( 1 + C_{d,p,\alpha}'
\left[
A^{b_\ast} \norme[C^{1-\alpha}]{\mathbf{v}} 
+ A^{\mu_\ast}
\right] \frac{s}{r^{\alpha}}  \right).
\]
The conclusion is identical, provided that $A$ is large enough and that the choice of $K$ and $\delta$ ensures
\begin{equation}\label{largeA2bis}
K \geq \frac{\alpha}{\omega} \left\{ \delta + C_{d,p,\alpha}' \left[ A^{b_\ast}
\norme[C^{1-\alpha}]{\mathbf{v}} 
+ A^{\mu_\ast}
\right]
\right\}
\end{equation}
instead of~\eqref{largeA2}.
Note that to reconcile \eqref{largeA2bis} with \eqref{largeA1} for large $A$, one needs  $\alpha>\omega \vee (1-\omega)$,
which is always possible if $\alpha>1/2$. However, when $\alpha\leq 1/2$, one needs one final modification, which is to replace
the average~$\bar{\mathbf{v}}_{\ball{x(s)}{r}}$ by the point-wise value $\mathbf{v}(x(s))$, where:
\begin{equation}\label{trackCenterBis}
\begin{cases}
x^\prime(s)=\mathbf{v}(x(s)),\\
x(0) = x_0.
\end{cases}
\end{equation}
In this case, estimate~\eqref{improvedVminusVBar}
is improved one step further into the following one:
\begin{equation}
\left|\mathbf{v}(x)-\mathbf{v}(x(s))\right| \leq
|x-x(s)|^{1-\alpha} \norme[C^{1-\alpha}]{\mathbf{v}}.
\end{equation}
This changes $J_0$ into
\[
\widetilde{J_0} = \int_{\ball{x(s)}{r}} |x-x(s)|^{-(\alpha-\omega)} |\psi|,
\]
which is then estimated in an identical manner  to $I_2$. Note that this modification also allows us to drop all requirements
concerning $b_\ast$ and in particular~\eqref{restrictP}, which is beneficial for any $\alpha\in(0,1)$.
Let us finally point out that, in the other parts of the proof, the requirement $d>\alpha$ now allows for any dimension $d\geq1$.
We have thus established the following statement:
\begin{prop}\label{PropagationAtomBis}
Let us assume that $0<\alpha<1$ and $d\geq 1$ and that the velocity field $\mathbf{v} \in C^{1-\alpha}$ satisfies
\begin{equation}\label{smallnessAssumption2}
(p-1) \norme[L^\infty_t L^{d/\alpha}_x]{(\div\mathbf{v})_-}  \leq S_{\alpha/2}(\R^d)
\end{equation}
for some $p=2^n$ with $n\in\N$ such that
\[
p > \frac{d}{d-(\alpha-\omega)} \qquad\text{with}\qquad
0<\omega<\alpha.
\]
Then there exist constants $\delta, K$ and $\gamma$ (and a lower threshold for $A$), depending only on $d$, $p$, $\alpha$
and~$\|\mathbf{v}\|_{C^{1-\alpha}}$, such that for all $r\in (0,1]$, the following implication holds:
\begin{equation}
\psi_0\in \atom^p_{r}
\qquad\Longrightarrow\qquad
\forall s>0,\qquad
\psi(s,\cdot)\in
\left(1+\frac{Ks}{r^\alpha}\right)^{-\delta /K} 
 \atom_{(r^\alpha+Ks)^{1/\alpha}}
\end{equation}
where $\psi$ denotes the solution of the Cauchy problem~\eqref{TDalphaCons}.
\end{prop}

\begin{rmk}
Note that we can take $p=2$ in the previous statement (and thus, using remark~\ref{choiceOfp}, claim a similar one in the case of $\T^d$) if
\[
\frac{d}{2} > \alpha-\omega.
\]
Such a choice is always possible.
\end{rmk}

\section{Proof of Theorem~\ref{mainThm}}\label{par:proof}

The proof of theorem~\ref{mainThm} is now straightforward.

\bigskip
Given  $d\geq2$ and $1\leq \alpha\leq 2$ (with $d\geq 3$ when $\alpha=2$), one chooses $\omega\in(0,1)$ such that $\alpha-1 < \omega<1$ if $\alpha<2$,
or $\omega>1/2$ if $\alpha=2$.
One checks immediately that $\omega<\alpha<d$ and $d>2(\alpha-\omega)$.
Let us now consider an advection vector field $v\in\BMO$ with
\[
\norme[L^{d/\alpha}]{(\div v)_-} \leq S_{\alpha/2}.
\]
One takes $p=2$.
One chooses the constant $A$, which is implicit in the definition of atoms, according to the threshold mentioned in remark~\ref{choiceOfA};
this threshold depends solely on $d$, $\alpha$ and $\norme[\BMO]{v}$. 
One considers the constants $\gamma$, $\delta$ and $K$ given by  propositions~\ref{PropagationAtom} and~\ref{PropagationAtomGlobal}
and sets
\[
\beta = \alpha\delta/K.
\]
The value of $\beta$ depends on $d$, $\alpha$ and $\norme[\BMO]{v}$. 

\bigskip
For $d\geq 1$ and $0<\alpha<1$, one chooses $\omega$ such that $(\alpha-\frac{d}{2})_+<\omega<\alpha$ and $p=2$. In this case,
the~$\BMO$ norm is replaced by the $C^{1-\alpha}$-norm in all computations.

\begin{rmk}\label{relax}
When $\norme[L^{d/\alpha}]{(\div v)_-} < 2 S_{\alpha/2}$,
one can still run the following proof. However, the choice of~$A$ and of all constants then depends not only on 
$\norme[\BMO]{v}$ but also on $C(v)=2 S_{\alpha/2} - \norme[L^{d/\alpha}]{(\div v)_-}>0$
and degenerates as $C(v)\to0$. See remark~\ref{LargerV}.
\end{rmk}

\subsection{Propagation of the H\"older regularity}

For any solution~$\theta$ of~\eqref{TDalpha} stemming from $\theta_0\in C^\beta$ and for $\psi_0\in\atom^2_r$, identity~\eqref{transfert} implies that:
\[
\int_{\R^d} \theta(t,x) \psi_0(x) dx  = \int_{\R^d} \theta_0(x) \psi(t,x) dx
\]
where $\psi$ is solution of the dual equation~\eqref{dualEQ}, which, by proposition~\ref{PropagationAtomGlobal}, is an atom of calibrated size.
Using~\eqref{atomCbeta} for $\theta_0$, one gets:
\[
r^{-\beta} \left| \int_{\R^d} \theta(t,x) \psi_0(x) dx \right|
\lesssim  r^{-\beta}\left(1+\frac{Kt}{r^\alpha}\right)^{-\delta /K}  (r^\alpha+Kt)^{\beta/\alpha}\norme[C^\beta]{\theta_0}  = \norme[C^\beta]{\theta_0}.
\]
A second application of~\eqref{atomCbeta} then ensures that $\theta(t)\in C^\beta$ and that
\[
\norme[C^\beta]{\theta(t)} \leq C \norme[C^\beta]{\theta_0}.
\]
The constant $C$ is the implicit one in~\eqref{atomCbeta}.
\begin{rmk}
The same argument also gives $\norme[C^{\beta'}]{\theta(t)} \leq C \norme[C^{\beta'}]{\theta_0}$ for any $0\leq \beta'\leq \beta$.
\end{rmk}

\subsection{Gain in H\"older regularity}
One can use the H\"older inequality and proposition~\ref{AtomL1LpBound} with $p=2$ to control
\[
\left| \int_{\R^d} \theta_0(x) \psi(t,x) dx \right| \leq 
\norme[L^q]{\theta_0} \norme[L^{q'}]{\psi(t)} \leq 
\frac{ A^{2/q} r^\beta \norme[L^q]{\theta_0} }{ (r^\alpha + K t)^{(\beta+\frac{d}{q})/\alpha} }
\lesssim r^\beta  t^{-(\beta+\frac{d}{q})/\alpha}  \norme[L^q]{\theta_0} 
\] 
for any Lebesgue exponent $q$ such that $2\leq q\leq \infty$.
One thus gets a regularization estimate:
\[
\norme[C^\beta]{\theta(t,\cdot)}
\simeq \sup_{\substack{0<r\leq 1\\ \psi_0\in \atom^2_r}} r^{-\beta} \left| \int_{\R^d} \theta(t,x) \psi_0(x) dx \right| 
\leq C t^{-(\beta+\frac{d}{q})/\alpha}  \norme[L^q]{\theta_0} 
\]
with a constant $C$ that depends on $q$ and $A$ and thus ultimately on $d$, $\alpha$ and $\norme[\BMO]{v}$.


\appendix
\section{Appendix: on the $L^1$-control of atoms}

Even without the a-priori constraint $\norme[L^1]{\psi}\leq 1$,
one can control the $L^1$-norm of atoms (or any $L^q$ norm for $q\leq p$) by a real interpolation estimate.
\begin{prop}\label{AtomL1LpBound2}
If $\varphi$ satisfies
\[
\norme[L^p]{\varphi} \leq A r^{-d(1-\frac{1}{p})},
\qquad\text{and}\qquad
\exists x_0\in\R^d, \quad \int_{\R^d} |\varphi(x)| \Omega(x-x_0) dx \leq r^\omega
\] 
for some $0<r\leq1$ and $p\in (1,\infty]$, then
\begin{equation}\label{AtomL1Bound2}
\norme[L^1]{\varphi} \leq C_{d,p} A^{\omega/(\omega+d(1-\frac{1}{p}))}
\end{equation}
and, more generally, for any $1\leq q\leq p$, one has
\begin{equation}\label{AtomLpBound2}
\norme[L^q]{\varphi} \leq C_{d,p,q} A^{\frac{\omega+d(1-1/q)}{\omega+d(1-1/p)}} r^{-d(1-\frac{1}{q})}.
\end{equation}
\end{prop}
\begin{rmk}
Compared to proposition~\ref{AtomL1LpBound}, these estimates ``lose'' powers of $A$,
which would provoke a critical collision of exponents in the previous proof (see remark~\ref{troubleWithL1}).
\end{rmk}

\begin{proof}
For any $\rho\in [0,r]$, one has
\[
\norme[L^1]{\varphi}
\leq \int_{\ball{x_0}{\rho}} |\varphi| +  \rho^{-\omega}\int_{\R^d\backslash\ball{x_0}{\rho}} |\varphi(x)| \Omega(x-x_0) dx
\leq A \left(\frac{\rho}{r}\right)^{d(1-\frac{1}{p})}  |\ball{0}{1}|^{1-\frac{1}{p}}+ \left( \frac{r}{\rho} \right)^{\omega}
\]
and \eqref{AtomL1Bound2} follows from choosing the optimal value $\rho=r \big( A |\ball{0}{1}|^{1-\frac{1}{p}} \big)^{-1/(\omega+d(1-\frac{1}{p}))}$.
For the second estimate, one proceeds similarly with $\tau\in [0,r]$; for clarity, we do not track the constant related to $|\ball{0}{1}|$:
\begin{align*}
\int_{\R^d} |\varphi|^{q} &\leq \int_{\ball{x_0}{\tau}} |\varphi|^{q}
+  \int_{\R^d \backslash \ball{x_0}{\tau}} |\varphi|^{q} \\
& \lesssim \biggl(\int_{\R^d} |\varphi|^{p}\biggr)^{\frac{q}{p}} \tau^{d(1-\frac{q}{p})}
+ \biggl(\tau^{-\omega}\int_{\R^d \backslash \ball{x_0}{\tau}} |\varphi(x)| \Omega(x-x_0) dx\biggr)^{\frac{p-q}{p-1}}
\biggl(\int_{\R^d} |\varphi|^{p}\biggr)^{\frac{q-1}{p-1}},
\end{align*}
thanks to the H\"older inequality (with $p/q\geq1$) for the first term,
and the interpolation inequality $\norme[L^q]{f}\leq \norme[L^1]{f}^{1-\theta} \norme[L^p]{f}^{\theta}$
with $\theta = (1-\frac{1}{q})/(1-\frac{1}{p}) \in [0,1]$ for the second.
We now use the fact that $\varphi\in \atom^p_r$ and deduce that
\begin{equation*}
\int_{\R^d} |\varphi|^{q}
\lesssim \left(A r^{-d(1-\frac{1}{p})} \right)^{q} \tau^{d(1-\frac{q}{p})} 
+ \biggl(\frac{r}{\tau}\biggr)^{\omega (\frac{p-q}{p-1})} (A^p r^{-(p-1)d})^{\frac{q-1}{p-1}}.
\end{equation*}
The optimal choice for $\tau$ is the one that balances the weight of both terms; it is $\tau=r A^{-p/(d(p-1)+\omega p)}$.
The computation then boils down to
\[
\int_{\R^d} |\varphi|^{q}\lesssim r^{-(q-1)d}A^{\frac{d p (q-1) + \omega pq}{d(p-1)+\omega p}}
\qquad\ie\qquad
\norme[L^q]{\varphi} \lesssim A^{\frac{\omega+d(1-1/q)}{\omega+d(1-1/p)}} r^{-d(1-\frac{1}{q})}
\]
and the lemma is proven.
\cqfd\end{proof}


\begin{thebibliography}{XX}\small

\bibitem{AT2017}
L.~Ambrosio, D.~Trevisan.
\textit{Lecture notes on the DiPerna-Lions theory in abstract measure spaces.}
Ann. Fac. Sci. Toulouse, XXVI, 4 (2017), 729-766.

\bibitem{BJ2018}
D.~Bresch, P.-E.~Jabin.
\textit{Quantitative regularity estimates for compressible transport equations.}
In: Bul\'{i}\v{c}ek M., Feireisl E., Pokorn\'{y} M. (eds)
New Trends and Results in Mathematical Description of Fluid Flows.
Ne\v{c}as Center Series. Birkh\"auser, Cham, 2018.

\bibitem{LBL2019}
C.~Le~Bris, P.-L.~Lions.
\textit{Parabolic equations with irregular data and related issues.}
Ser. App. Num. Math. 4. De Gruyter, 2019.

\bibitem{CV2010}
L.A.~Caffarelli, A.~Vasseur.
\textit{Drift diffusion equations with fractional diffusion and the quasi-geostrophic equation.}
Ann. Maths. 171, 3 (2010) 1903-1930.

\bibitem{CCW2001}
P.~Constantin, D.~Cordoba and J.~Wu.
\textit{On the critical dissipative quasi-geostrophic equation.}
Indiana Univ. Math. J., 50 (2001), 97-107.

\bibitem{CGHV2014}
P.~Constantin, N.~Glatt-Holtz and V.~Vicol.
\textit{Unique ergodicity for fractionally dissipated, stochastically forced 2d-Euler equations.}
Comm. in Math. Physics 330, (2014) 819-857.

\bibitem{CTV2015}
P.~Constantin, A.~Tarfulea, V.~Vicol.
\textit{Long time dynamics of forced critical SQG.}
Communications in Mathematical Physics 335, 1  (2015), 93-141.

\bibitem{CC2004} 
A.~Cordoba, D.~Cordoba.
\textit{A maximum principle applied to quasi-geostrophic equations.}
Commun. Math. Phys. 249 (2004), 511-528.

\bibitem{D2010}
M.~Dabkowski.
\textit{Eventual regularity of the solutions to the supercritical dissipative quasi-geostrophic equation.}
Geom. Func. Anal. 21, 1 (2011), 1-13.
2010.

\bibitem{DL2007}
C.~De Lellis.
\textit{Notes on hyperbolic systems of conservation laws and transport equations.}
Handbook of Differential Equations: Evolutionary Equations, Volume 3, 2007, 277-382.

\bibitem{DipL1989}
R.J.~DiPerna, P.L.~Lions.
\textit{Ordinary differential equations, transport theory and Sobolev spaces.}
Invent. math. 98 (1989), 511-547.

\bibitem{GV2015}
F.~Golse, A.~Vasseur.
\textit{Hölder Regularity for Hypoelliptic Kinetic Equations with Rough Diffusion Coefficients.}
arXiv:1506.01908 (2015).

\bibitem{ISV2016}
C.~Imbert, R.~Shvydkoy and F.~Vigneron.
\textit{Global Well-Posedness of a Non-local Burgers Equation: the periodic case.}
Ann. Fac. Sci. Toulouse, XXV, 4 (2016), 723-758.

\bibitem{KN2009}
A.~Kiselev, F.~Nazarov.
\textit{Variation on a theme of Caffarelli and Vasseur.}
Journal of Mathematical Sciences 166 -- 1 (2010), 31-39.

\bibitem{KNV2007}
A.~Kiselev, F.~Nazarov and A.~Volberg.
\textit{Global well-posedness for the critical 2D dissipative quasi-geostrophic equation.}
Invent. Math. 167, 3 (2007), 445-453.

\bibitem{Kru1970}
S.N.~Kru\v{z}kov.
\textit{First order quasilinear equations in several independent variables.}
Mat. USSR Sbornik 10, 2 (1970), 217-243.

\bibitem{M2018}
C.~Mouhot.
\textit{De Giorgi-Nash-Moser and Hörmander theories: new interplay.}
arXiv:1808.00194 (2018).

\bibitem{Silv2010}
L.~Silvestre.
\textit{H\"older estimates for advection fractional-diffusion equations.}
Ann. Sc. Norm. Super. Pisa Cl. Sci. (5) Vol. XI (2012), 843-855.

\bibitem{SV2018}
L.~Silvestre, V.~Vicol.
\textit{H\"older continuity for a drift-diffusion equation with pressure}.
Ann. I.H.P. Analyse non linéaire 29, 4 (2012), 637-652.

\bibitem{STEIN1993}
E.M.~Stein.
\textit{Harmonic Analysis.}
Princeton University Press, 1993.

\bibitem{T1979}
G.~Talenti.
\textit{Best Constant in Sobolev Inequality.}
Annali di Matematica Pura ed Applicata (1979).


\bibitem{Vaz2017}
J.~V\'{a}zquez.
\textit{The mathematical theories of diffusion. Nonlinear and fractional diffusion.}
Vol. 2186 of Lecture Notes in Math., Springer, 2017.
 
\end{thebibliography}
\end{document}